\documentclass[a4paper,12pt]{amsart}

\usepackage{amsmath}
\usepackage{amssymb}
\usepackage{amsthm}
\usepackage{amscd}
\usepackage{geometry}

\newtheorem{theorem}{Theorem}[section]
\newtheorem{proposition}[theorem]{Proposition}

\newtheorem{corollary}[theorem]{Corollary}
\newtheorem{definition}[theorem]{Definition}
\newtheorem{remark}[theorem]{Remark}

\newcommand{\C}{\mathbb{C}}
\newcommand{\Z}{\mathcal{Z}}
\newcommand{\holo}{\mathcal{O}}
\newcommand{\dbar}{\bar{\partial}}
\newcommand{\Hom}{\operatorname{Hom}}

\newcommand{\J}{\mathcal{J}}
\newcommand{\I}{\mathcal{I}}

\title{An effective uniform Artin-Rees lemma}
\author{Johannes Lundqvist}
\date{}
\address{\newline Department of Mathematics\newline Stockholm University\newline SE-106 91 Stockholm\newline Sweden}
\email{johannes@math.su.se}

\begin{document}

\maketitle

\begin{abstract}
We prove a global uniform Artin-Rees lemma type theorem for sections of ample line bundles over smooth projective varieties. This result is used to prove an Artin-Rees lemma for the polynomial ring with uniform degree bounds. The proof is based on multidimensional residue calculus.
\end{abstract}

\section{Introduction}\label{Intro}

Assume that $(X,x)$ is a germ of a reduced analytic variety. Let $M$ be a finitely generated module over the local ring, $\holo_{X,x}$, of germs of holomorphic functions at $x$. In \cite{Sznajdman} it was proved by residue calculus that if $N$ is a submodule of $M$, then there exists a constant $\mu$ such that the inclusion
\begin{equation}\label{AR}
I^{\mu+r}M\cap N\subset I^r N
\end{equation}
holds for all ideals $I$ of $\holo_{X,x}$ and all non-negative integers $r$.
This is the well-known uniform Artin-Rees lemma that was proved by Huneke in \cite{Huneke} for much more general rings.

The uniform Artin-Rees lemma is related to the theorem of Brian\c con-Skoda, \cite{BriSko}. Since there are global versions of the latter, see \cite{EL} and \cite{Hic} for smooth $X$ and \cite{AW2} for singular $X$,
it is reasonable to believe that there is a global version of the inclusion \eqref{AR}.
In this paper we prove such a result when $X$ is smooth.

\begin{theorem}\label{MainThm}
Assume that $X$ is a smooth projective variety of dimension $n$ and that $L$ is an ample line bundle over $X$.
Assume moreover that $f^1,\ldots,f^m$ are global holomorphic sections of $L$. Then there exist constants $\mu$ and $s_0$ such that for every set of global holomorphic sections $g^1,\ldots,g^{\ell}$ of any ample line bundle $M$ over $X$ the following is true: If $\phi$ is a global section of 
\begin{equation*}
M^{\otimes s}\otimes K_X \otimes L^{\otimes s_0}, \quad s\geq n + r,\quad r\geq 1,
\end{equation*}
such that $\phi\in \J(f)$ and $|\phi|\leq C|g|^{\mu+r-1}$ for some $C>0$, then
\begin{equation*}
\phi = \sum_{\substack{j=1,\ldots,m\\I_1 + \ldots + I_{\ell}= r }} \alpha_{I,j} (g^1)^{I_1} \ldots (g^{\ell})^{I_{\ell}} f^j,
\end{equation*} 
where $\alpha_{I,j}$ are global sections of $M^{\otimes (s-r)}\otimes K_X \otimes L^{\otimes (s_0-1)}$.
\end{theorem}
Here and throughout this paper $|g|$ is short for $|g^1|+\ldots+|g^{\ell}|$.
\begin{remark}\label{rmk}{\rm
We may replace the canonical bundle $K_X$ in Theorem~\ref{MainThm} with any bundle $T$ such that $T\otimes K_X^{-1}$ is non-negative. This follows from the proof in Section~\ref{MainSec}.
}
\end{remark}

By the theorem of Brian\c con-Skoda,
\[
|\phi|\leq C |g|^{\mu+r+n-2}
\]
implies that $\phi\in\J(g)^{\mu+r-1}$, and this certainly implies that $|\phi|\leq C^{\prime}|g|^{\mu+r-1}$.
Since $\mu$ is not specified in general we might as well use such an estimate instead of the membership condition.
We choose to use the inequality in this paper for purely technical reasons.
Also, we actually get a special case of the theorem of Brian\c{c}on-Skoda from Theorem \ref{MainThm} with this setting.

If we assume that $M=L$ and $r=1$ we get the following result.
\begin{corollary}\label{coro}
Assume that $f^1,\ldots, f^m$ and $L$ are as in Theorem~\ref{MainThm}.
Then there exist constants $\mu$ and $s_0$ such that for every set of global holomorphic sections $g^1,\ldots,g^{\ell}$ of $L$ the following holds: If $\phi$ is a global section of $K_X\otimes L^{\otimes s_0}$, that satisfies $\phi\in \J(f)$ and $|\phi|\leq C|g|^{\mu}$, then
\begin{equation}\label{thm1form}
\phi=\sum_{ij}\alpha_{ij} g^i f^j,
\end{equation}
where $\alpha_{ij}$ are global sections of $K_X \otimes L^{\otimes (s_0-2)}$.
\end{corollary}

\begin{remark}\label{BST}
{\rm
If $\J(f)=\J(1)$, then it follows from the proof in Section~\ref{MainSec} that we may take $\mu$ in Corollary \ref{coro} as $\min(n,\ell)$ and we get back a theorem of Brian\c con-Skoda type, cf. part (ii) of Corollary 2.2 in \cite{EL} and  Theorem~7.1, and its proof, in \cite{AW2}.
That is, assume that
$X$ and $L$ are as in Theorem~\ref{MainThm} and $g^1,\ldots,g^{\ell}$ are global holomorphic sections of $L$. Then if $\phi$ is a global section of 
\[
K_X\otimes L^{\otimes s},\quad s\geq n+1,
\]
such that $|\phi|\leq C|g|^{\min(n,\ell)}$, we may write
\[
\phi=\sum_{j}\alpha_{j} g^j,
\]
where $\alpha_{j}$ are global sections of $K_X \otimes L^{\otimes (s-1)}$.
}
\end{remark}

Based on Theorem~\ref{MainThm} and a geometric inequality in \cite{EL}
we prove a theorem about polynomials, which can be regarded as an effective uniform Artin-Rees lemma for the polynomial ring. 
\begin{theorem}\label{poly}
Let $V\subset\C^N$ be an algebraic variety of dimension $n$ and assume that $X$, the closure of $V$ in $\mathbb{P}^N$, is smooth. 
Given polynomials $F_1,\ldots,F_m$ on $V$ there exists a constant $\mu$ such that the following holds:
Assume that $G_1,\ldots, G_{\ell}$ are polynomials on $V$ of degree at most $d$, $r$ is a positive integer, and $\Phi$ is a polynomial such that
\begin{equation}\label{olikhet}
|\Phi|\leq C|G|^{\mu+r-1}
\end{equation}
and
\[
\Phi\in J(F_1,\ldots,F_m).
\]
Then there exist polynomials $P_{I,j}$ such that
\begin{equation*}
\Phi = \sum_{\substack{j=1,\ldots,m \\ I_1+\cdots +I_{\ell}=r}} P_{I,j} G_1^{I_1} \ldots G_{\ell}^{I_{\ell}} F_j,
\end{equation*}
and
\begin{align}
&\deg(P_{I,j} G_1^{I_1} \ldots G_{\ell}^{I_{\ell}} F_j)\leq \nonumber\\ 
&\max\left( (\mu+r-1) d^{c^G_{\infty}} \deg X + \deg \Phi, (n+r)d+ \kappa_1, \deg \Phi + \kappa_2 \right)\label{max},
\end{align}
where the constants $\kappa_1$ and $\kappa_2$ only depend on $J(F)$ and $V$.
\end{theorem}
Here $J(F)$ is the polynomial ideal generated by $F_1,\ldots,F_m$.
The constant $c_{\infty}^G$ is defined in Section \ref{polysec}; it is less than or equal to $n$. From this result we also derive a similar but weaker result in the case when $X$ is singular, see Section~\ref{sing}.

If $X=\mathbb{P}^n$, $\ell=1$, and $G_1=1$, then \eqref{max} becomes $\deg{\Phi}+\kappa$ for some $\kappa$. It is well known that in general $\kappa$ is double exponential in the degree of the $F_j$:s, \cite{Mayr}, and it was proved already in \cite{Hermann} that one can choose $\kappa$ as something like $2(2d^{\prime})^{2^N-1}$, where $d^{\prime}\geq \deg F_j$. This shows that the third entry in \eqref{max} is not only there for technical reasons. The same is true for the other entries as well. Assume for example that $r=1$ and that the zero set of $J(G)$ does not intersect the hyperplane at infinity. In this case $c_{\infty}^G=-\infty$. However, if we let $d$ tend to infinity it must be the case that the degree of $P_{i,j}G_iF_j$ tends to infinity linearly, so the second entry is necessary.
Now, consider the case when $r=1$, $J(F)=J(1)$, and assume that the zero set of $J(G)$ is empty. Then it was proved by Koll\'ar, \cite{Kollar}, Sombra, \cite{Sombra}, and Jelonek, \cite{Jelonek}, that in general the degree of $P_{i,j}G_iF_j$ cannot be chosen less than $d^{\min(\ell,n)}$, so we need something like the first entry.

In special cases one can explicitly calculate the degree estimates and get back classical theorems of Macaulay and Max Noether. This is discussed in the end of  Section \ref{polysec}.\\
\\
{ \bf Acknowledgement: } The author would like to thank Mats Andersson and Elizabeth Wulcan for valuable discussions and comments throughout the writing process of this paper.

\section{Andersson-Wulcan currents and the diamond product}\label{CurrDiam}

In this section we describe a residue current, introduced in \cite{AW}, associated to a generically exact Hermitian complex of vector bundles and also an operation on such complexes introduced in \cite{Sznajdman}.\\

Assume that $E_j$ are Hermitian vector bundles over an $n$-dimensional smooth variety $X$ in $\mathbb{P}^{N}$ and that the complex
\begin{equation}\label{HermCompl}
\ldots \overset{f_2}{\longrightarrow} E_2 \overset{f_2}{\longrightarrow} E_1 \overset{f_1}{\longrightarrow} E_0
\end{equation}
is generically exact, i.e., pointwise exact outside some proper analytic subvariety, $\Z$, of $X$. Let $E=\bigoplus E_k$.
Then there is a natural superstructure, i.e., a $\mathbb{Z}_2$-grading, on $E$ 
, see \cite{AW}. From now on and throughout this paper we assume that $E$ 
is equipped with that superstructure.
Consider the sheaves, $\mathcal{E}^{p,q}(E)$, of smooth $(p,q)$-forms on $X$ with values in $E$ and the space, $\mathcal{D}^{\prime}(E)$, of currents with values in $E$.
The operator
\[
\nabla_E=\sum f_j-\dbar
\] acts on $\mathcal{E}^{p,q}(E)$ and is naturally extended to $\mathcal{D}^{\prime}(E)$ and
the superstructure on $E$ 
makes sure that $\nabla_E^2
=0$, see \cite{AW}.

If $\sigma_k$ is the minimal inverse to $f_k$ on $X\setminus\Z$, i.e,
\[
\sigma_k \xi = 
\begin{cases}
\eta,\text{ where } f_k\eta=\xi \text{ and }\eta \text{ has minimal norm, if } \xi\in \operatorname{Im}f_k,\\
0, \text{ if } \xi \in(\operatorname{Im}f_k)^{\perp},
\end{cases}
\]
then the $\Hom(E_0,E)$-valued form
\[
u:=\sigma_1+\sigma_2\dbar\sigma_1+\sigma_3\dbar\sigma_2\dbar\sigma_1+\ldots
\]
satisfies
\[
\nabla_E u = 1_{E_0},
\]
see \cite {AW}. Note that the component 
\[
u_k := \sigma_k\dbar\sigma_{k-1}\cdots\dbar\sigma_1
\] 
of $u$ that takes values in $\Hom(E_0,E_k)$ has bidegree $(0,k-1)$.
The form $u$ can be extended across $\mathcal{Z}$ to a current $U$ by letting 
\begin{equation}\label{U}
U := \lim_{\epsilon \to 0} \chi (|h|^2/\epsilon^2)u,
\end{equation}
where $h_1,\ldots,h_M$ are functions with $\Z$ as their common zero set.
Here $\chi(t)$ is a smooth function on the reals that is $0$ for $t<1$ and $1$ for $t>2$. 
The existence of the limit \eqref{U} is nontrivial and requires the desingularization theorem of Hironaka. 

We now define the residue current
\begin{equation}\label{R}
R: = 1_{E_0} - \nabla_E U.
\end{equation}
\noindent
It obviously has support on $\Z$. The current $R$ is also a so-called pseudomeromorphic current as defined in \cite{AWdecomp}. We may restrict such currents to subvarieties in the following way. If $T$ is a pseudomeromorphic current on $X$ and $V$ is a subvariety of $X$ then the restriction of $T$ to the complement of $V$ has a natural extension to $X$, denoted $1_{V^c}T$. The difference between the current $T$ and that extension is a current with support on $V$ denoted $1_VT$. That is,
\begin{equation}\label{decomp}
T = 1_{V}T + 1_{V^{c}}T.
\end{equation}
For details, see \cite{AWdecomp}.

The sheaf complex
\begin{equation}\label{SheafCompl}
\ldots \overset{f_2}{\longrightarrow} \holo(E_2) \overset{f_2}{\longrightarrow} \holo(E_1) \overset{f_1}{\longrightarrow} \holo(E_0),
\end{equation}
 associated to the complex \eqref{HermCompl}, plays a key role in the following basic result, \cite{AW}.

\begin{theorem}\label{dualitet}
Assume that $X$ is smooth and that $E_0$ in the complex \eqref{HermCompl} has rank one.
Let $\mathcal{J}$ be the ideal sheaf $\operatorname{Im}(f_1)$ of the associated sheaf complex.
If $\phi$ is a holomorphic section of $E_0$, then $\phi\in \mathcal{J}$ if $R\phi = 0$, and the converse is true if the associated sheaf complex is exact.
\end{theorem}

Notice that even if the complex \eqref{HermCompl} is infinite the residue only takes values in $\Hom(E_0,E_0\oplus\ldots\oplus E_{\dim(X)+1})$. This follows from the construction of $u$ since the component $u_k$ has bidegree $(0,k-1)$.

We would like to use Theorem~\ref{dualitet} to draw the conclusion that a given section belongs to a certain product ideal. In order to do so we need an appropriate complex like \eqref{HermCompl} such that $\operatorname{Im}(f_1)$ lies in the product ideal in question. 
We use a construction due to \cite{Sznajdman} and we give here the definition and basic properties.

\begin{definition}
Given $r$ Hermitian complexes $E^1_{\bullet},\ldots, E^r_{\bullet}$, with morphisms $f^k_j:E^k_j\to E^k_{j-1}$, the diamond product, denoted $ E^1_{\bullet}\lozenge\ldots\lozenge E^r_{\bullet}$, is the complex $H_{\bullet}$, where
\begin{equation*}
H_0 = E_0^1\otimes\ldots\otimes E_0^r,\quad H_k = \bigoplus_{\substack{\alpha_1+\cdots+\alpha_r\\=k-1}}E_{1+\alpha_1}^1\otimes \cdots\otimes E^r_{1+\alpha_r}, 
\end{equation*}
and where the maps $h_j:H_j\to H_{j-1}$ are defined as 
\begin{equation*}
h_1 = f_1^rf_1^{r-1}\ldots f_1^1,\quad h_k = \sum_{1\leq s\leq r, j\geq 2} f_j^s{}\big|_{H_k}.
\end{equation*}
\end{definition}
Note that it follows directly from the definition that 
\begin{equation}\label{ordning}
E^1_{\bullet}\lozenge E^2_{\bullet}\lozenge E^3_{\bullet}=(E^1_{\bullet}\lozenge E^2_{\bullet})\lozenge E^3_{\bullet}=E^1_{\bullet}\lozenge (E^2_{\bullet}\lozenge E^3_{\bullet}).
\end{equation}

If $r$ is odd, then $ E^1_{\bullet}\lozenge\ldots\lozenge E^r_{\bullet}$ inherits its superstructure from the superstructures of the complexes $E_k$. However, if $r$ is even, then one needs to do a trick by
multiplying with the trivial complex 
\[
0\to E\to E\to0,
\]
for any bundle $E$. For details, see \cite{Sznajdman}.

Let $u^k$ be the $\Hom(E_0^k,E^k)$-valued form associated to the complex $E_{\bullet}^k$. It was shown in \cite{Sznajdman} that the form 
\begin{equation}\label{udelar}
u^H:=u^1\otimes\cdots\otimes u^r
\end{equation}
satisfies the equality
\[
\nabla_H u = 1_{H_0}.
\]
From $u^H$ we define the currents $U^H$ and $R^H$ as in \eqref{U} and \eqref{R}.
One can describe the residue current $R^H$ in terms of the individual building block complexes. Assume that $H_{\bullet}$ is the diamond complex of $M_{\bullet}$ and $L_{\bullet}$ and assume that $U^L,R^L,U^M$ and $R^M$ are the currents associated to $L_{\bullet}$ and $M_{\bullet}$. Assume also that $L_{\bullet}$ is exact outside an analytic set defined by a tuple, $h_1$, of analytic functions and let $h_2$ be a tuple that defines the corresponding set for $M_{\bullet}$. Then
\begin{equation}\label{SplitRes}
R^H=R^M\wedge U^L - U^M\wedge R^L,
\end{equation}
where 
\begin{align*}
R^M\wedge U^L &= \lim_{\epsilon\to 0} \dbar \chi (|h_2|^2/\epsilon^2)\wedge u^M\wedge U^L \\
&= \lim_{\epsilon\to 0}\lim_{\delta\to 0} \dbar  \chi (|h_2|^2/\epsilon^2)\wedge u^M\wedge\chi (|h_1|^2/\delta^2)  u^L,
\end{align*}
and 
\begin{align}
U^M\wedge R^L &= \lim_{\epsilon\to 0} \chi (|h_2|^2/\epsilon^2)u^M\wedge R^L \label{temp}\\
&= \lim_{\epsilon\to 0}\lim_{\delta\to 0} \chi (|h_2|^2/\epsilon^2) u^M\wedge \dbar  \chi (|h_1|^2/\delta^2)\wedge u^L,\nonumber
\end{align}
see Proposition~3.4 in \cite{Sznajdman}.

Products of more than two factors are defined in the same way.
Once again, the existence of the limits is non-trivial.
The order of the limits is important as we see in the one-variable principal value example 
\[U=\frac{1}{z},\quad R=\dbar\frac{1}{z}.\]
In this case we get
\[
U\wedge R = 0,\quad R\wedge U = \dbar\frac{1}{z^2}.
\]

\section{The proof of Theorem~\ref{MainThm}}\label{MainSec}

Our proof of Theorem~\ref{MainThm} is based on the fact that $\phi$ annihilates a residue current $R^H$ associated to the diamond product of appropriate choices of complexes.

Let $X,L,M,f^j$ and $g^j$ be as in Theorem~\ref{MainThm}. Since $L$ is ample there exists an exact sequence like \eqref{SheafCompl}, with a direct sum of negative powers of $L$ as $E_k$, such that $\operatorname{Im}f_1=\J(f)$ , see for example \cite{Lazar1}. Indeed,
consider the sequence 
\[
\oplus\holo(L^{-1})\overset{f}{\longrightarrow}\holo_X\longrightarrow \holo_X/\J(f)\longrightarrow 0,
\]
where $f$ is the mapping $(f^1,\ldots,f^m)$.
Let $F$ be the kernel of the surjection $f$. Then 
$F\otimes\holo(L^{\otimes d_2})$ is generated by its global sections if $d_2$ is big enough by the Cartan-Serre-Grothendieck theorem. Fixing generating sections we get a surjective map $\holo_X\to F\otimes\holo(L^{\otimes d_2})$ and hence we have a surjection $\holo(L^{-\otimes d_2})\to F$. If we repeat this argument for the kernel of that map and so on we get a, possibly non-terminating, exact complex
\begin{equation*}
\ldots\overset{f_3}{\longrightarrow}\oplus\holo(L^{-\otimes d_2})\overset{f_2}{\longrightarrow}\oplus\holo(L^{-1})\overset{f_1=f}{\longrightarrow} \holo_X \longrightarrow \holo_X/J(f) \longrightarrow 0,
\end{equation*}
where $d_2,d_3\ldots$ are positive integers. For a Hermitian vector bundle $S_0$ we get a Hermitian complex
\begin{equation}\label{Fcompl}
\ldots\overset{f_3}{\longrightarrow} S_0\otimes(\oplus L^{-\otimes d_2})\overset{f_2}{\longrightarrow}S_0\otimes (\oplus L^{-1})\overset{f}{\longrightarrow}S_0,
\end{equation}
that is pointwise exact outside the zero set of $\J(f)$.

For $\J(g)$ we choose the Koszul complex, i.e.,
we let $E^j$ be trivial line bundles over $X$ with global frames $e_j$ and set
\[
E = M^{-1}\otimes E^1 \oplus \ldots \oplus M^{-1}\otimes E^l.
\]
Then
the Koszul complex
is the Hermitian complex
\begin{equation}\label{Gcompl}
0\longrightarrow E_n \overset{\delta_n}{\longrightarrow}\ldots\overset{\delta_2}{\longrightarrow}E_1 \overset{\delta_1}{\longrightarrow}E_0,
\end{equation}
where
\begin{equation*}
E_k = 
\Lambda^k E 
= 
M^{-k}\otimes\Lambda^k(E^1\oplus\ldots\oplus E^{\ell}).
\end{equation*}
The maps $\delta_k: E_k\to E_{k-1}$ are interior multiplication with the section $g$ of $E^*$, where $g=\sum g^je_j^*$ and $e^*_j$ is the dual frame. For details, see for example Example 2.1 in \cite{AW2}.

Denote the complex \eqref{Fcompl} by $L_{\bullet}$ and by $M_{\bullet}$ the Koszul complex associated to $\J(g)$. For a Hermitian line bundle $S$ let $R^H$ be the residue current from Section~\ref{CurrDiam} associated to the complex 
\begin{equation}\label{H}
H_{\bullet} := (S\otimes\underbrace{ M_{\bullet} \lozenge M_{\bullet} \lozenge \ldots \lozenge M_{\bullet}}_{r \text{ times}})\lozenge L_{\bullet}.
\end{equation}
Then, according to \eqref{SplitRes} and \eqref{ordning}, we can write 
\[
R^H = R^M \wedge U^L - U^M \wedge R^L,
\]
where $R^L,U^L,R^M$ and $R^M$ are the currents associated to the complexes $L_{\bullet}$ and $S\otimes M_{\bullet}\lozenge \ldots \lozenge M_{\bullet}$.

The following proposition from \cite{AW2} can be seen as a global version of the first part of Theorem~\ref{dualitet}.
\begin{proposition}\label{kohomologi}
Assume that $\eqref{HermCompl}$ is a generically exact Hermitian complex over a smooth variety $X$ and that $\phi$ is a holomorphic section of the bundle $E_0$. If $R$ is the associated residue current, $R\phi=0$, and 
\begin{equation*}
H^{k-1}(X,\holo(E_k))=0, \quad 1\leq k \leq n+1,
\end{equation*}
then there is a global holomorphic section $\psi$ of $E_1$ such that $f^1\psi=\phi$.
\end{proposition}

We are now in a position to prove Theorem~\ref{MainThm}.

\begin{proof}[Proof of Theorem \ref{MainThm}]
Assume that $\phi\in \J(f)$.
Let  $H_{\bullet}$ be the complex \eqref{H} and
choose $S_0$ as $K_X\otimes L^{\otimes s_0}$ and $S$ as $M^{\otimes s}$ in \eqref{Fcompl} and \eqref{H}, respectively.
If we can prove that $R^H\phi=0$, then $\phi$ would be on the form \eqref{thm1form} by Proposition~\ref{kohomologi} if all the relevant cohomology groups vanish.

We are interested in the cohomology groups of the bundles $H_k$ in $H_{\bullet}$ for $1\leq k\leq n+1$. Remember that
$H_k$ consists of a sum of tensor products of one bundle from the complex \eqref{Fcompl} and $r$ bundles from \eqref{Gcompl} tensored by $S$. The possible bundles from $\eqref{Fcompl}$ are 
\[
S_0\otimes L^{-d_{j}},\quad 1\leq j \leq k,
\] 
and the possible bundles from \eqref{Gcompl} are
\[
M^{-j}\otimes\Lambda^{j}(E^1\oplus\ldots\oplus E^{\ell}),\quad 1\leq j \leq k.
\]
Note that the exponent of $M$ in $H_1$ is $s-r$ and that the exponent decreases by at most $1$ at every level in $H_{\bullet}$. In particular,
since a tensor product of ample bundles is ample we can use Kodaira's vanishing theorem to see that the relevant cohomology groups vanish if
\[
s_0\geq\max_{1\leq j \leq n+1}d_j+1=d_{n+1}+1
\]
and
\[
s\geq n+r.
\]

Fix $s_0$ and $s$ so that all the cohomology groups vanish.
It then remains to show that there exists a constant $\mu$ such that $\phi$ annihilates the residue $R^H$, given that $|\phi|\leq C |g|^{\mu + r-1}$.
Remember that $R^H$ splits into the sum 
\begin{equation}\label{uppdelning}
R^M\wedge U^L - U^M\wedge R^L.
\end{equation} 
Since $\phi$ is assumed to belong to $\J(f)$ we get that $R^L\phi =0$ by the second part of Theorem~\ref{dualitet}, and in view of \eqref{temp} $U^M\wedge R^L \phi= 0$.

To see that the first term in \eqref{uppdelning} is annihilated
we use that there exists a modification  $\widetilde{X}\overset{\pi}{\longrightarrow} X$ so that the pull back of $U^L$ locally can be expressed as a finite sum of forms
\[
\pi_*(\frac{smooth}{h}\pi^*\phi),
\]
where $h$ is a section to a line bundle $\widetilde{L}$ over $\widetilde{X}$ such that it locally is a monomial in some local coordinates, see \cite{AW}.
In light of \eqref{udelar} we hence get that locally $R^M\wedge U^L\phi$ is the limit of the pushforward of a finite sum of terms on the form
\begin{equation}\label{InnanPartial}
\pi^*(\dbar \chi(|g|^2/{\epsilon^2})\wedge (u^1\otimes\ldots\otimes u^r))\wedge\frac{smooth}{h}\pi^*\phi,
\end{equation}
where every $u^j$ is associated to $M_{\bullet}$.
Since $X$ is compact the divisor of $h$ is a finite sum $\sum \tau_j D_j$ for positive integers $\tau_j$ and if $\tau=\sum \tau_j$ we get that $h$ locally is a monomial of degree less than or equal to $\tau$ at every point in $X$.
The arguments after expression (4.10) in the proof of Theorem~1.2 in \cite{Sznajdman} now show that $R^M\wedge U^L \phi = 0$, locally at a point $x$, if $\mu\geq \min(\ell,n) + \tau+1$. 
Since $n$ and $\tau$ do not depend on $g$ or $x$ the conclusion of the theorem follows if 
\[
|\phi|\leq C |g|^{\mu+r-1},
\]
where $\mu\geq \min(\ell,n) + \tau+1$.
\end{proof}

\begin{remark}\label{BS}{\rm
Note that if the $f_j$:s do not have any common zeros, i.e., $\J(f)=\J(1)$, then $\tau=0$ and we may choose $\mu$ as $\min(\ell,n)+1$. 
If one carefully reads the proof of Theorem 1.2 in \cite{Sznajdman} one sees that $\mu=\min(\ell,n)$ does the trick in this case. We then get the result in Remark~\ref{BST}.
}
\end{remark}

\section{The proof of Theorem~\ref{poly}}\label{polysec}

Let $X$ be a smooth projective variety of dimension $n$, $L$ a nef line bundle over $X$ and $\J\in\holo_X$ an ideal sheaf. If $Z_j$ are the distinguished subvarieties in the sense of Fulton-MacPherson of $\J$  , see \cite{EL},
and $r_j$ are the coefficients associated to the $Z_j$:s, then
\begin{equation}\label{GeoOlikhet}
\sum r_j \deg_L Z_j \leq \deg_L X,
\end{equation} 
where
\begin{equation*}
\deg Z_j = \int_{Z_j} c_1(L)^{\dim Z_j}
\end{equation*}
is the $L$-degree of $Z_j$.
The geometric inequality \eqref{GeoOlikhet} above is proved in \cite[Proposition~3.1]{EL}.
Note that if $L=\holo(d)$, then 
\begin{equation}\label{degL}
\deg_L X = d^{n}\deg X,
\end{equation}
where $\deg X$ denotes $\deg_{\holo(1)} X$.

If $g_j$ is the $d$-homogenization of $G_j$, then for the ideal sheaf $\J(g)$ we associate a number $c^G_{\infty}$ defined to be the maximal codimension of the distinguished subvarieties $Z_j$ contained in the hyperplane at infinity. 
If there is no distinguished subvariety at infinity we assign to  $c^G_\infty$ the value $-\infty$.
Using \eqref{degL} and \eqref{GeoOlikhet} we get that if $L=\holo(d)$, then
\begin{equation}\label{colik}
r_j \leq d^{c^G_\infty}\deg X
\end{equation}
for $r_j$ associated to $Z_j$ contained in the hyperplane at infinity.

\begin{proof}[Proof of Theorem~\ref{poly}]
Let $V, X, G_j, F_j$ and $\Phi$ be as in Theorem~\ref{poly}.
Let $d^{\prime}$ be the maximum of the degrees of all the polynomials $F_j$ and let $f_j$ and $g_j$ be the $d^{\prime}$ and $d$-homogenization of $F_j$ and  $G_j$, respectively. 
Let
\begin{equation}\label{homoPhi}
\phi = z_0^{\rho-\deg \Phi}\Phi(z_0/z)z_0^{\deg{\Phi}}
\end{equation}
be the $\rho$-homogenization of $\Phi$.
We consider $f_j$ and $g_j$ as sections of $\holo(d^{\prime})$ and $\holo(d)$ restricted to $X$.
The bundle $K_X^{-1}\otimes\holo(k)$ is ample for $k$ large enough, say $k\geq k_X$. By Remark~\ref{rmk} we may therefore use Theorem~\ref{MainThm} on $\phi$ if $\rho$ is big enough, $\phi$ belongs to $\J(f)$ even at the hyperplane at infinity, and the inequality 
\begin{equation}\label{olikbev}
|\phi|\leq C |g|^{\mu+r-1}
\end{equation}
is valid on the whole of $X$.
Let us first show that $\phi$ belongs to $\J(f)$ provided that $\rho$ is larger than some constant depending on $F_1,\ldots,F_m$ and $V$. If $R^f$ is the residue associated to a locally free resolution of $\J(f)$, then by the second part of Theorem~\ref{dualitet} we only need to prove that $R^f$ is annihilated by $\phi$. 
Remember that we may write
\begin{equation}\label{uppdelbev}
R^f = 1_VR^f + 1_{X\setminus V}R^f,
\end{equation}
cf., Section \ref{CurrDiam}.
Since $\phi\in \J(f)$ on $V$ it follows from Theorem \ref{dualitet} that $\phi$ annihilates $1_VR^f$. We know that 
$1_{X\setminus V}R^f$ has support on the hyperplane at infinity so $z_0^\nu$ annihilates $1_{X\setminus V}R^f$ if $\nu$ is large enough, say larger than $\nu_f$.
This means that if $\rho$ in \eqref{homoPhi} is chosen so that 
\begin{equation}\label{ett}
\rho\geq\deg\Phi+\nu_f,
\end{equation}
then $R^f$ is annihilated by $\phi$ and thus $\phi\in\J(f)$.

To make sure that \eqref{olikbev} holds we consider the normalization 
\[
\widetilde{X}\overset{\pi}{\longrightarrow}X,
\]
of the blow-up of $X$
along $\J(g)$.
Let $X_\infty$ be the part of $X$ that intersect the hyperplane at infinity and
write the exceptional divisor as $W=\sum r_j W_j$. 
Then, by definition, the distinguished subvarieties $Z_j$ are the images of $W_j$, and hence
\[
r_j\leq d^{c_\infty^G}\deg X
\] 
if $W_j\subseteq X_\infty$
by \eqref{colik}.
The polynomial $\Phi$ satisfies \eqref{olikhet} by hypothysis so we get that $\pi^*\phi$ vanishes to order $(\mu + r-1) r_j$ on $W_j$ if $\pi W_j \nsubseteq X_{\infty}$. If $\pi W_j\subseteq X_{\infty}$, then $\pi^*\phi$ vanishes to order $\rho - \deg\Phi$ on $W_j$.
If we choose $\rho$ such that
\begin{equation}\label{tva} 
\rho\geq (\mu+r-1)d^{c_\infty^G}\deg X+\deg\Phi,
\end{equation}
we get that $\pi^* \phi$ vanishes to order $(\mu+r-1) r_j$ on all $W_j$. This means that $|\pi^*\phi|\leq C |\pi^* g|^{\mu+r-1}$ on the whole of $\widetilde{X}$ and hence \eqref{olikbev} holds.

If also
\begin{equation}\label{tre}
\rho \geq d(n+r)+ (d^{\prime}+k_X)s_0,
\end{equation}
where $s_0$ is the same as the one in Theorem $\ref{MainThm}$ we may apply that theorem on $\phi$ with 
\[M=\holo(d)\big|_X, \quad L=\holo(d^{\prime}+k_X)\big|_X.\]

To sum up, we may use Theorem~\ref{MainThm} if $\rho$ satisfies the inequalities
\eqref{ett}, \eqref{tva}, and \eqref{tre}.
The only thing left is that we need to make sure that the sections $\alpha_{I,j}$ that we get after applying Theorem \ref{MainThm} have extensions to global sections of $\holo(\rho)$. However, that is true if $\rho$ is larger than an absolute number $\eta$ depending on $X$.
The theorem follows with $\kappa_1=(d^{\prime}+k_X)s_0$ and $\kappa_2=\nu_f +\eta$.
\end{proof}

If $V=\C^n$ and hence $X=\mathbb{P}^n$ so that $\deg X=1$
and moreover $J(F)=J(1)$ and $r=1$, then it follows from the proof of Theorem~\ref{MainThm} and Theorem~\ref{poly} that $\kappa_2=\kappa_1=0$. However, one can actually take $\kappa_1=-n$. To see this we just modify the proof of Theorem~\ref{MainThm} slightly. Instead of taking $S=\holo(sd)$ we could take $S=\holo(s)$. In this case we get that $s$ should be so large so that the cohomology groups $H^j(\mathbb{P}^n,\holo(s-d(n+1)))$ vanishes. From Kodaira's vanishing theorem we see that  $s\geq d(n+1)-n$ does the trick. Together with Remark~\ref{BS} we get the following effective version of the Brian\c con-Skoda theorem.

\begin{theorem}\label{Nullstellen}
For every set of polynomials $G_1,\ldots,G_{\ell}$ on $\mathbb{C}^n$ with degree less than or equal to $d$ the following holds:
If $\Phi$ is a polynomial such that $|\Phi|\leq C|G|^{\min(\ell,n)}$, then there exist polynomials $P_j$ such that
\begin{equation*}
\Phi = P_1 G_1 + \ldots +P_{\ell} G_{\ell},
\end{equation*}
and the degree of $P_jG_j$ is at most 
\[
\max\left( \min(\ell,n)d^{c^G_{\infty}} + \deg\Phi, (n+1)d - n\right).
\]
\end{theorem}

The theorem above was already proved in \cite{AG}. Note that if we also assume that the common zero set is empty we almost get back the optimal degree estimate, $d^{\min(\ell,n)}$, of Koll\'ar and Jelonek, mentioned in Section~\ref{Intro}.
If we also assume that $G_1,\ldots,G_{\ell}$ have no common zeros at infinity we do get back the classical theorem of Macaulay, \cite{Macaulay}.
That is, we may write 
\begin{equation*}
1=\sum P_jG_j,
\end{equation*}
where the degree of $P_jG_j$ is at most $(n+1)d-n$.

If we assume that $\deg G_j = 0$, the common zero set of $F_1,\ldots,F_m$ is a discrete set, $m=n$, and that there are no zeros at the hyperplane at infinity, then we get back the theorem of Max Noether, i.e., we may write \begin{equation*}
\Phi=\sum P_jF_j,
\end{equation*}
where the degree of $P_jF_j$ is at most $\deg\Phi$, \cite{Max}.
To see this we first note that $c_\infty^G=-\infty$ and that $\kappa_2=0$. 
This means that $\deg P_j F_j\leq\max(\deg\Phi,\kappa_1)$. From the proof of Theorem~\ref{poly} we know that $\kappa_1$ is a multiple of $s_0$ from Theorem~\ref{MainThm}. In this case this means that $\kappa_1$ is a number so that $H^{k-1}(\mathbb{P}^n,\holo(\kappa_1-d_kd^{\prime}))=0$, where $d_k$ are the numbers in the proof of Theorem~\ref{MainThm} and $d^{\prime}$ is the maximum degree of the $F_j$:s.
Since $\J(f)$ is a complete intersection we may use the Koszul complex as the exact sequence that defines the residue associated with $\J(f)$. In particular, it has length $n$ which means that we may choose $\kappa_1$ as $0$.

\section{The non-smooth case}\label{sing}

Let $V\subset \mathbb{C}^N$ be a singular reduced algebraic variety of dimension $n$. It was noted by Mats Andersson that one can deduce an Artin-Rees lemma type result on $V$ from the smooth case, i.e., Theorem~\ref{poly}:

\begin{theorem}\label{oglatt}
Let $V$ be as above and let $F_1,\ldots,F_m$ be polynomials on $V$.
Then there exist constants $\mu$ and $\nu$ such that the following holds: Assume that $G_1,\ldots,G_{\ell}$ are polynomials of degree at most $d$ and that $\Phi$ is a polynomial such that 
\begin{equation}\label{oglattolik}
|\Phi|\leq |G|^{\mu+\nu}
\end{equation}
and 
\begin{equation*}
\Phi\in(F_1,\ldots, F_m)
\end{equation*}
on $V$.
Then there exist polynomials $A_{i,j}$ such that
\begin{equation*}
\Phi = \sum A_{i,j}G_iF_j
\end{equation*}
on $V$ and
\[\deg(A_{j,\ell}G_{j}F_{\ell})\leq \deg\Phi + (\nu+\mu) d^n\deg X + \mu d^N + O(d).\]
\end{theorem}
The degree estimate in this result is of type $O(d^N)$ and not as expected of type $O(d^n)$. It is probably true that there is an estimate of type $O(d^n)$ but we cannot prove any such result at this time.

\begin{proof}
Let $F_1\ldots,F_m$ be polynomials on $V\subset \mathbb{C}^N$, let $X$ be the closure of $V$ in $\mathbb{P}^N$, and let $H_1,\ldots,H_t$ cut out $V$, i.e., $J_V=(H_1,\ldots,H_t)$. 

First, Theorem~\ref{poly} implies that there exists a constant $\mu$ such that for every set of polynomials $G_1,\ldots,G_{\ell}$ in $\mathbb{C}^N$ and every polynomial $\widehat{\Phi}$ in $\C^N$ we have that
\begin{align}\label{muu}
&|\widehat{\Phi}|\leq |G|^{\mu},\quad\widehat{\Phi}\in(F_1,\ldots,F_m,H_1\ldots,H_t) \\
&\Longrightarrow \widehat{\Phi} = \sum A_{j,\ell}G_jF_{\ell}+\sum B_{j,\ell}G_jH_{\ell}, \nonumber
\end{align}
where $A_{j,\ell},B_{j,\ell}$ are polynomials and
\[ \deg A_{j,\ell}G_jF_{\ell}\leq \deg \widehat{\Phi} + \mu d^N + O(d).\]

Second, there is a Brian\c{c}on-Skoda-Huneke constant $\nu$ on $V$, see \cite[Theorem~6.4]{AW2}, such that if $\Phi$ and $G_1,\ldots,G_{\ell}$ are as in Theorem~\ref{oglatt} and \eqref{oglattolik} holds on $V$, then
\[
\Phi = \sum_{|I|=\mu}a_IG^I
\]
on $V$ with
\[
\deg a_IG^I \leq \deg\Phi + (\nu+\mu) d^n\deg X + O(d).
\]

Consider
\[\widehat{\Phi} = \sum_{|I|=\mu} a_IG^I\]
as a polynomial in $\mathbb{C}^N$.
Then clearly $|\widehat{\Phi}|\leq |G|^{\mu}$
in $\C^N$ and moreover, $\widehat{\Phi}=\Phi$ on $V$ which means that $\widehat{\Phi}\in(F_1,\ldots,F_m,H_1\ldots,H_t)$. Therefore, by Theorem~\ref{poly} as above we get that 
\[
\widehat{\Phi}=\sum A_{i,j}G_iF_{j}+\sum B_{i,j}G_iH_{j},
\]
with \[\deg(A_{i,j}G_iF_{j})\leq \deg{\widehat{\Phi}}+\mu d^N +O(d).\]
This means that 
\[\Phi = \sum A_{ij}G_{i}F_{j}\]
on $V$ with
\[\deg(A_{i,j}G_{i}F_{j})\leq \deg\Phi + (\nu+\mu) d^n\deg X + \mu d^N + O(d).\]
Note that the linear term $O(d)$ is independent of $\Phi$ and the polynomials $G_1,\ldots,G_{\ell}$.
\end{proof}

\end{document}